\documentclass[10pt,a4paper]{article}
\usepackage[utf8]{inputenc}
\usepackage{amsmath}
\usepackage{amsfonts}
\usepackage{amssymb}
\usepackage{amsthm}
\usepackage{color}
\usepackage{enumerate}
\usepackage{marvosym}
\usepackage{hyperref}
\usepackage{comment}
\usepackage{multicol}
 \newtheorem{thm}{Theorem}[section]

 \newtheorem*{thm*}{Theorem}
 \newtheorem{cor}[thm]{Corollary}
 \newtheorem{lem}[thm]{Lemma}
 \newtheorem{prop}[thm]{Proposition}
 \theoremstyle{definition}
 \newtheorem{defn}[thm]{Definition}
 \theoremstyle{remark}
 \newtheorem{rem}[thm]{Remark}

 \newtheorem{ex}{Example}
 
 \numberwithin{equation}{section}
 
\newcommand{\vertiii}[1]{{\left\vert\kern-0.25ex\left\vert\kern-0.25ex\left\vert #1
  \right\vert\kern-0.25ex\right\vert\kern-0.25ex\right\vert}}

\newcommand{\sslash}{\mathbin{/\mkern-6mu/}}

\newcommand{\VERT}[1]{{\left\vert\kern-0.25ex\left\vert\kern-0.25ex\left\vert #1 
    \right\vert\kern-0.25ex\right\vert\kern-0.25ex\right\vert}}

\begin{document}
\title{Commutative C$^\ast$ algebras and Gelfand theory through phase space methods}
\author{Robert Fulsche, Oliver F\"{u}rst}
\date{\today}

\maketitle
\begin{abstract}
    We show how the Gelfand spectrum of certain commutative operator algebras can be studied based on the theorem of Stone and von Neumann. The method presented is a natural addition to the tools of \emph{quantum spectral synthesis}, which were recently used to characterize certain commutative Toeplitz algebras on the Fock space. Our method applies to this setting and also to more general abelian phase spaces. Besides characterizing Gelfand spectra of such commutative operator algebras, we also prove an extension of this result to the operator-valued case.
\end{abstract}



\section{Introduction}
Finding abelian subalgebras of certain non-commutative operator algebras is a classical problem in the field of operator theory and operator algebras. In recent years, this problem motivated significant research, for example in the field of Toeplitz operators (see, e.g., \cite{Bauer_Rodriguez2022, Vasilevski_2008}). A recurring theme is relating commutativity of certain subalgebras with invariance under actions of certain groups. Among many of the different interesting settings that one can investigate, the commutative subalgebras of the Toeplitz algebra on the Fock space is a frequent choice \cite{Esmeral_Maximenko2016, Esmeral_Vasilevski2016, Dewage_Olafsson2022}. Recently, tools from the theory of \emph{quantum harmonic analysis} \cite{Werner1984, Fulsche2020, Luef_Skrettingland2018} were used to investigate commutativity of Toeplitz algebras on the Fock space \cite{Dewage_Mitkovski2023, Dewage_Mitkovski2024, Fulsche_Rodriguez2023}. In particular, in \cite{Fulsche_Rodriguez2023} the methods of \emph{quantum spectral synthesis} were used by the first-named author and Miguel Rodriguez to characterize certain classes of commutative Toeplitz algebras affiliated with Lagrangian subgroups of the phase space $\Xi = \mathbb R^{2n}$, extending previously known results \cite{Esmeral_Vasilevski2016}. After identifying these commutative algebras, their Gelfand theory was described. In \cite{Fulsche_Rodriguez2023}, this was done by a certain ad-hoc approach. Indeed, the Gelfand theory of the characterized commutative algebras can also be described by a suitable application of the Stone-von Neumann theorem. The present note is supposed to describe this approach in detail. 

In \cite{Fulsche_Rodriguez2023}, it turned out that the commutative algebras that were found are affiliated with Lagrangian subgroups of the phase space. As, for example, described in \cite{Mumford1991}, Lagrangian subgroups of phase spaces serve also another purpose: To set up certain canonical representations of the CCR relations. By the Stone-von Neumann theorem, this canonical representation is unitarily equivalent to the one considered before. As it turns out, this unitary equivalence maps the commutative algebras under consideration to algebras of multiplication operators. Hence, this precisely yields a description of the Gelfand theory.

To increase the significance of the present work, we present these results for general abelian phase spaces (instead of just $\mathbb R^{2n}$); the setting of quantum harmonic analysis with respect to such phase spaces was recently discussed in \cite{Fulsche_Galke2023}. As the methods from \cite{Fulsche_Rodriguez2023} extend verbatim to this larger class of phase spaces, using the tools presented in \cite{Fulsche_Galke2023}, we find it appropriate to discuss our present results in this more general setting. Nevertheless, we will not repeat much of the theory discussed in \cite{Fulsche_Rodriguez2023}. Indeed, most of what we present in the present work can be understood without reference to the techniques from our previous works.

Let us explain the contents of the previous note in some more details: We will usually consider $\Xi$ to be a locally compact abelian (lca) group, usually written additively. By $\widehat{\Xi}$ we will denote the Pontryagin dual of $\Xi$. Let $m: \Xi \times \Xi \to \mathbb T$ a multiplier on it. By this, we mean a (measurable) map satisfying
\begin{align*}
    m(x+y,z)m(x,y) &= m(x,y+z)m(y,z),\\
    m(x,0) = m(0,x) &= 1,
\end{align*}
for all $x, y, z \in \Xi$. For such a multiplier, we define the map $\sigma: \Xi \times \Xi \to \mathbb T$ by
\begin{align*}
    \sigma(x,y) = \frac{m(x,y)}{m(y,x)}.
\end{align*}
It is not hard to see that $\sigma$ is always an alternating bicharacter.
\begin{defn}
    The tuple $(\Xi, m)$ consisting of the lca group $\Xi$ and the multiplier $m$ is said to be an \emph{abelian phase space} if the map
    \begin{align*}
        \Xi \ni x \mapsto \sigma(\cdot, x) \in \widehat{\Xi},
    \end{align*}
    is an isomorphism of lca groups.
\end{defn}
The theorem of Stone-von Neumann (and Mackey-Baggett-Kleppner) states the following:
\begin{thm}{\cite[Theorem 3.3]{Baggett_Kleppner1973}}
    Let $(\Xi, m)$ be a phase space. Then, there exists a unique (up to unitary equivalence) irreducible projective unitary representation $(\mathcal H, U)$ with $m$ as its multiplier, i.e., $U_x U_y = m(x,y) U_{x+y}$ for all $x, y \in \Xi$. 
\end{thm}
Let now $H \subset \Xi$ be a closed subgroup. We define
\begin{align*}
    H^\sigma = \{ z \in \Xi: ~\sigma(z, w) = 1, \text{ all } w \in H\}.
\end{align*}
$H^\sigma$ is easily seen to be again a closed subgroup. By basic results from Pontryagin duality, $H^{\sigma \sigma} = H$. $H$ is said to be \emph{Lagrangian} if $H^\sigma = H$. 

We set:
\begin{align}
    \mathcal L(\mathcal H)_H = \{ A \in \mathcal L(\mathcal H): ~U_h A = A U_h, \text{ all } h \in H\}.
\end{align}
In \cite{Fulsche_Rodriguez2023}, the method of (quantum) spectral synthesis was employed to prove the following result for the phase space $\Xi = \mathbb R^{2n}$, endowed with the multiplier $m((x, \xi), (y,\eta)) = \exp(-iy\xi)$. The same methods can, without significant difference, be employed to prove the following result. Again, we note that the methods of quantum harmonic analysis on general abelian phase spaces have been discussed in \cite{Fulsche_Galke2023}. Based on the same arguments presented there, one can prove:
\begin{thm}\label{thm:Fulsche_Rodriguez}
    Let $(\Xi, m)$ be an abelian phase space such that the unique irreducible $m$-representation $(\mathcal H, U)$ is square integrable. 
    Let $H \subset \Xi$ be a closed subgroup. Then, $\mathcal L(\mathcal H)_H$ is a maximal translation invariant abelian von Neumann subalgebra if and only if $H$ is Lagrangian.
\end{thm}
As mentioned before, the Stone-von Neumann theorem allows for a description of the Gelfand theory of $\mathcal L(\mathcal H)_H$ by a method more general (and, depending of the taste of the reader, also more elegant) than the one utilized in \cite{Fulsche_Rodriguez2023}. The main goal of this note is describing this and obtaining the following result:
\begin{thm}\label{thm:main}
    Let $(\Xi, m)$ be an abelian phase space and $H \subset \Xi$ a Lagrangian subgroup, which is first-countable. Further, let $(\mathcal H, U)$ be the unique irreducible $m$-representation of $\Xi$. Then, $\mathcal L(\mathcal H)_H$ is a commutative von Neumann algebra with $\mathcal L(\mathcal H)_H \cong L^\infty(\Xi / H)$. 
\end{thm}
Let us emphasize that, in this result, we do not need to assume that the representation $(\mathcal H, U)$ is square-integrable. Hence, even commutativity of $\mathcal L(\mathcal H)_H$ in the above result does, in general, not follow from Theorem \ref{thm:Fulsche_Rodriguez}. We want to note that the assumption of first-countability of $H$ is of technical nature. We will discuss this later at an appropriate point.

Note that the algebra
\begin{align*}
    \mathcal C_1(\mathcal H) := \{ A \in \mathcal L(\mathcal H): ~\| \alpha_z(A) - A\| \to 0 \text{ as } z \to 0\},
\end{align*}
plays an important role in quantum harmonic analysis (see, e.g., \cite{Werner1984, Fulsche2020, Fulsche_Galke2023, Fulsche_Luef_werner2024}).  In the setting of $\Xi = \mathbb R^{2n}$, $\mathcal C_1(\mathcal H)$ agrees with the $C^\ast$-algebra generated by all Toeplitz operators on the Segal-Bargmann space with bounded symbols \cite{Fulsche2020}. Hence, it is natural to ask about the Gelfand theory of $\mathcal C_1(\mathcal H)_H := \mathcal C_1(\mathcal H) \cap \mathcal L(\mathcal H)_H$. This is described by the following result, which is an immediate consequence of the above result:
\begin{cor}\label{corollary}
    Let $(\Xi, m)$ be an abelian phase space and $H \subset \Xi$ a first-countable Lagrangian subgroup. Then, $\mathcal C_1(\mathcal H)_H \cong \operatorname{BUC}(\Xi / H)$.
\end{cor}

Besides proving the above theorem and its corollary, which will be done in Section \ref{sec:2} of this note, we will also provide an extension of the theorem to the operator-valued case (Section \ref{sec:3}). Finally, in Section \ref{sec:examples}, we will discuss several examples, where we apply our results.

\section{The main theorem}\label{sec:2}

We note that when $H$ is a Lagrangian subgroup of $\Xi$, then $\sigma|_{H \times H} = 1$. This is of course equivalent to $m(x,y) = m(y,x)$ for all $x, y \in H$. In this case, one can show that $m(x,y) = \frac{\alpha(x+y)}{\alpha(x)\alpha(y)}$ for all $x, y \in H$ and some Borel map $\alpha: H \to \mathbb T$ with $\alpha(0) = 1$ (cf.\ \cite[p.~308]{Baggett_Kleppner1973}).

 For describing the Gelfand theory of $\mathcal L(\mathcal H)_H$ we will consider a certain canonical representation of $(\Xi, m)$ affiliated with the Lagrangian subgroup $H$. Mumford describes this representation in \cite{Mumford1991}, but does not give all the technical details. We will provide enough details to keep things self-contained.

Therefore, we fix in the following a phase space $(\Xi, m)$ and a Lagrangian subgroup $H \subset \Xi$.
\begin{defn}
    The space $L^2(\Xi \sslash H)$ is defined to be the space consisting of all (equivalence classes of) measurable functions $f: \Xi \to \mathbb C$ satisfying
    \begin{align*}
        f(x+h) = \overline{\alpha(h) m(h,x)} f(x), \quad x \in \Xi, h \in H,\\
        \int_{\Xi/H} |f(x+H)|^2~d(x+H) < \infty.
    \end{align*}
    Note that the first condition implies that $|f(x+h)| = |f(x)|$ such that the second condition is well-defined.
\end{defn}
Since for $f, g \in L^2(\Xi \sslash H)$, $f \cdot \overline{g}$ is also $H$-invariant, the standard $L^2(\Xi / H)$-inner product defines an inner product on $L^2(\Xi \sslash H)$.

\begin{defn}
    On $L^2(\Xi \sslash H)$ define
    \begin{align*}
        U_x f(t) = m(t,x) f(t+x).
    \end{align*}
\end{defn}
\begin{lem}
 For each $x \in \Xi$, $U_x$ is a unitary operator on $L^2(\Xi \sslash H)$, satisfying $U_x^\ast = \overline{m(x,-x)}U_{-x}$. Further, for $x, y \in \Xi$ we have $U_x U_y = m(x,y) U_{x+y}$. 
\end{lem}
\begin{proof}
Verification of these facts is straightforward.
\end{proof}
What is indeed not trivial to check are the following two points: The representation acts irreducibly on $L^2(\Xi \sslash H)$. Even worse, given that it is not straightforward to construct nonzero elements of $L^2(\Xi \sslash H)$ explicitly, it is not even clear if $L^2(\Xi \sslash H) \neq \{ 0\}$. Both points will be clarified later.

Let $f \in L^\infty(\Xi / H)$. Then, multiplication by $f$, i.e., the operator $M_f$ defined by
\begin{align*}
    M_f g(x) = f(x+H) g(x), \quad g \in L^2(\Xi \sslash H),
\end{align*}
leaves $L^2(\Xi \sslash H)$ invariant as is not hard to verify. Further, one immediately sees that $\| M_f\| \leq \| f\|_\infty$. Indeed, we even have equality, which is again not straightforward to see. 

Note that for $h \in H$:
\begin{align*}
    U_h M_f g(t) = m(t,h) f(t+h) g(t+h) = f(t) m(t,h) g(t+h) = M_f U_h g(t),
\end{align*}
i.e., $U_h M_f = M_f U_h$. Indeed, the converse to this statement is also true, which is a key step in proving our main result:
\begin{prop}
    Let $A \in \mathcal L(L^2(\Xi \sslash H))_H$. Then, there exists $f \in L^\infty(\Xi / H)$ such that $A = M_f$.
\end{prop}
The latter statement shows that, as von Neumann algebras, $\mathcal L(L^2(\Xi \sslash H))_H \cong L^\infty(\Xi / H)$, which then also determines the Gelfand spectrum of $\mathcal L(L^2(\Xi \sslash H))_H$ as that of the standard von Neumann algebra $L^\infty(\Xi / H)$.

The main result is based on the following lemma:
\begin{lem}\label{lemma:multiplier}
    Assume that there exists a Borel-measurable cross section $\gamma: \Xi/H \to \Xi$ (i.e., $\pi \circ \gamma = Id$ on $\Xi/H$, where $\pi: \Xi \to \Xi / H$ is the quotient map). Then, the map
    \begin{align*}
        \Phi: \Xi \to \mathbb T, \quad \Phi(x) := \overline{\alpha(x-\gamma(x+H)) m(x-\gamma(x+H), \gamma(x+H))},
    \end{align*}
    is measurable and satisfies
    \begin{align*}
        \Phi(x+h) = \frac{\Phi(x)}{\alpha(h)m(h,x)}, \quad x \in \Xi, ~h \in H.
    \end{align*}
\end{lem}
\begin{proof}
    Measurability of $\Phi$ follows from measurability of $\gamma, \alpha$ and $m$. We only need to verify that the functional equation is satisfied. For this, note that $\gamma(x+h+H) = \gamma(x+H)$. We therefore obtain the following equivalences of statements:
    \begin{align*}
        \Phi(x+h) &= \frac{\Phi(x)}{\alpha(h)m(h,x)}\\
        \Longleftrightarrow \frac{\alpha(x-\gamma(x+H)) \alpha(h) }{\alpha(x+h-\gamma(x+H))} &= \frac{m(x+h - \gamma(x+H), \gamma(x+H))}{m(x-\gamma(x+H), \gamma(x+H)) m(h,x)}.
        \intertext{Using that $\frac{\alpha(h)\alpha(x-\gamma(x+H))}{\alpha(x+h - \gamma(x+H))} = \overline{m(h, x-\gamma(x+H))}$ we obtain:}
        \Longleftrightarrow \frac{1}{m(h, x-\gamma(x+H))} &= \frac{m(x+h - \gamma(x+H), \gamma(x+H))}{m(x-\gamma(x+H), \gamma(x+H)) m(h,x)}.
        \end{align*}
        Using the cocycle identity to write
        \begin{align*}
        m(x-\gamma(x+H), \gamma(x+H)) = \frac{m(x, 0)m(-\gamma(x+H), \gamma(x+H))}{m(x, -\gamma(x+H))},
        \end{align*}
        as well as $m(x+h-\gamma(x+H), \gamma(x+H)) = \frac{m(x+h, 0)m(-\gamma(x+H), \gamma(x+H))}{m(x+h, -\gamma(x+H))}$, together with $m(0,y) = 1 = m(y,0)$ for every $y \in \Xi$, we obtain:
        \begin{align*}
        \Phi(x+h) & = \frac{\Phi(x)}{\alpha(h)m(h,x)}\\
        \Longleftrightarrow \frac{1}{m(h,x-\gamma(x+H)} & = \frac{m(x, -\gamma(x+H))}{m(x+h, -\gamma(x+H)m(h,x)}.
    \end{align*}
    Now, applying the identity $m(x+h, -\gamma(x+H)) = \frac{m(h,x-\gamma(x+H)) m(x, -\gamma(x+H))}{m(h,x)}$ shows that this last equality is satisfied, which finishes the proof.
\end{proof}

By the theorem of Birkhoff and Kakutani, the subgroup $H \subset \Xi$ is metrizable if and only if it is first-countable. In this case, \cite[Theorem 1]{Feldman_Greenleaf_1968} ensures that the quotient map $\pi: \Xi \to \Xi/H$ admits a Borel measurable cross section, i.e., there exists a Borel-measurable map $\gamma: \Xi / H \to \Xi$ such that $\pi \circ \gamma = Id$ on $\Xi / H$. With this $\gamma$ fixed and $\Phi$ as in Lemma \ref{lemma:multiplier}, we can now continue.
\begin{rem}
    The existence of cross sections of the canonical quotient map can be ensured also in the general case \cite{Kehlet1984, Kupka1983}, but only with weaker measurability properties. We think that the present note is not the right place for elaborated discussions on measurability properties of cross sections of the quotient map. Nevertheless, with ideas such as those presented in \cite{Kehlet1984, Kupka1983}, it seems reasonable to expect that the restriction of $H$ being first-countable can be overcome.
\end{rem}

Having established the existence of a map from Lemma \ref{lemma:multiplier}, we can now consider the operator $A_\Phi: L^2(\Xi/H) \to L^2(\Xi \sslash H)$, which is formally a multiplication operator, defined as follows:
\begin{align*}
    A_\Phi f(x) = \Phi(x) f(x+H).
\end{align*}
Then, we easily obtain that $A_\Phi$ is unitary, which we fix as a lemma:
\begin{lem}
    Let $H \subset \Xi$ be a first-countable Lagrangian subgroup and $\gamma: \Xi/H \to \Xi$ be a Borel measurable section of the quotient map $\pi: \Xi \to \Xi/H$. Then, the operator $A_\Phi: L^2(\Xi/H) \to L^2(\Xi \sslash H)$ defined above is a unitary operator. For $f \in L^2(\Xi \sslash H)$, the function $\overline{\Phi} \cdot f$ is $H$-invariant and $A_\Phi^\ast: L^2(\Xi \sslash H) \to L^2(\Xi / H)$ acts as $A_\Phi^\ast f(x+H) = [\overline{\Phi} \cdot f](x + H)$.
\end{lem}
The verification of this lemma is again straightforward. As an immediate consequence, we obtain the following facts:
\begin{prop}
    Let $(\Xi, m)$ be an abelian phase space and $H \subset \Xi$ a first-countable Lagrangian subgroup.
    \begin{enumerate}
        \item $L^2(\Xi \sslash H) \neq \{ 0\}$.
        \item If $f \in L^\infty(\Xi / H)$, then the multiplication operator $M_f: L^2(\Xi \sslash H) \to L^2(\Xi \sslash H)$ has operator norm $\| M_f \| = \| f\|_\infty$. 
    \end{enumerate}
\end{prop}
\begin{proof}
    The first statement now follows from $L^2(\Xi / H) \neq \{ 0\}$. The second statement is a consequence of $\| M_f\|_{L^2(\Xi / H) \to L^2(\Xi / H)} = \| f\|_\infty$ and $A_\Phi M_f = M_f A_\Phi$. 
\end{proof}

As the next step, we translate the action of $U_x$ from $L^2(\Xi \sslash H)$ to $L^2(\Xi / H)$. Direct computations show:

\begin{lem}
    Let $x \in \Xi$. Then, 
    \begin{align*}
        A_\Phi^\ast U_x A_\Phi f(t+H) = \frac{\Phi(t+x)}{\Phi(t)} m(t,x) f(t+x+H), \quad f \in L^2(\Xi / H).
    \end{align*}
    We note that this expression makes sense, as the map $t \mapsto \frac{\Phi(t+x)}{\Phi(t)} m(t,x)$ is $H$-invariant.
\end{lem}
Indeed, the above lemma shows why we prefer to work with $L^2(\Xi \sslash H)$ as a representing space, even though one could have expected to work on $L^2(\Xi / H)$ directly: For defining the representing operators on $L^2(\Xi / H)$, one would have to choose the Borel-measurable section $\gamma$ to set up the operators $U_x$. We find it more appropriate to define the representation without any reference to such a choice, and use $\gamma$ only as a tool for our method of proof.

The above expression for $A_\Phi^\ast U_x A_\Phi$ looks cumbersome to work with. Fortunately, it simplifies significantly when $x \in H$:
\begin{lem}
    Let $h \in H$. Then,
    \begin{align*}
        A_\Phi^\ast U_h A_\Phi f(t+H) = \frac{1}{\alpha(h)} \sigma(t+H, h) f(t+H), \quad f \in L^2(\Xi / H).
    \end{align*}
\end{lem}
\begin{proof}
    Clearly, we have $f(t+h+H) = f(t+H)$ for $t \in \Xi$ and $h \in H$. Hence, we only need to show that
    \begin{align*}
        \frac{\Phi(t+h)}{\Phi(t)} m(t,h) = \frac{1}{\alpha(h)} \sigma(t,h).
    \end{align*}
    But this is an immediate consequence of Lemma \ref{lemma:multiplier}. Hence, observing that the map $t \mapsto \sigma(t, h)$ is $H$-invariant finishes the proof.
\end{proof}
Recall, as a basic fact from the duality theory of locally compact abelian groups, that $\widehat{H} \cong \widehat{\Xi} / A(\widehat{\Xi}, H)$, where $A(\widehat{\Xi}, H)$ denotes the annihilator of $H$ in $\widehat{\Xi}$. Using this and the fact that $\Xi \ni z \mapsto \sigma(\cdot, z) \in \widehat{\Xi}$ is assumed to be an isomorphism of lca groups, we see that $A(\widehat{\Xi}, H) \cong H^\sigma = H$ such that $\widehat{H} \cong \Xi / H$. By applying the Pontryagin dual again, we see that every character of $\Xi / H$ is of the form $\chi(t+H) = \sigma(t+H, h)$ for some $h \in H$. Hence, we obtain:
\begin{lem}
    Let $(\Xi, m)$ be an abelian phase space and $H \subset \Xi$ a first-countable Lagrangian subgroup. Then,
    \begin{align*}
        \operatorname{span} \{ A_\Phi^\ast U_h A_\Phi: ~h \in H\} = \operatorname{span} \{ M_\chi \in \mathcal L(L^2(\Xi / H)): ~\chi \in \widehat{\Xi / H} \}.
    \end{align*}
\end{lem}

\begin{lem}
    Let $(\Xi, m)$ be an abelian phase space and $H \subset \Xi$ a first-countable Lagrangian subgroup. Then,
    \begin{align*}
        A_\Phi^\ast \mathcal L(L^2(\Xi \sslash H))_H A_\Phi = \{ A \in \mathcal L(L^2(\Xi / H)): ~M_\chi A = AM_\chi, ~\chi \in \widehat{\Xi / H}\}.
    \end{align*}
\end{lem}
Clearly, if $A$ commutes with every $M_\chi$, then $A$ commutes also with every operator in $\operatorname{span}\{ M_\chi: ~\chi \in \widehat{\Xi /H}\}$. Further, this span is dense, in weak operator topology, in $\{ M_f: ~f \in L^\infty (\Xi / H)\}$. From here, it is simple to conclude that
\begin{align*}
    A_\Phi^\ast \mathcal L(L^2(\Xi \sslash H))_H A_\Phi = \{ M_f \in \mathcal L(L^2(\Xi / H)): ~f \in L^\infty(\Xi / H)\}',
\end{align*}
the commutant of the multiplication operators. But the multiplication operators are well-known to form a factor (in the sense of von Neumann-algebras), i.e., the commutant is again simply the set of multiplication operators. Hence, we arrive at:
\begin{prop}\label{prop:multiplicationoperatorsintertwine}
    $\mathcal L(L^2(\Xi \sslash H))_H = \{ M_f \in \mathcal L(L^2(\Xi \sslash H)): ~f \in L^\infty(\Xi / H)\}$. As von Neumann-algebras, we therefore have:
\begin{align*}
    \mathcal L(L^2(\Xi \sslash H))_H \cong L^\infty(\Xi / H).
\end{align*}
\end{prop}

We now turn towards $\mathcal C_1(L^2(\Xi \sslash H))_H$. Observe that
\begin{align*}
    \mathcal C_1(L^2(\Xi \sslash H))_H =  \{ M_f \in \mathcal L(L^2(\Xi \sslash H)): ~\| \alpha_x(M_f) - M_f\|_{op} \to 0, ~x \to 0\}.
\end{align*}
\begin{lem}
    Let $f \in L^\infty(\Xi / H)$ and $x \in \Xi$. Then, $\alpha_x(M_f) = M_{f(\cdot + x)}$. 
\end{lem}
\begin{proof}
    This simply follows from $U_x^\ast = \overline{m(x,-x)} U_{-x}$ (which is readily verified) and:
    \begin{align*}
        U_x M_f U_x^\ast h(t) &= m(t,x) f(t+x+H) [U_x^\ast h](t+x)\\
        &= m(t,x) f(t+x+H) \overline{m(x,-x)} [U_{-x}]h(t+x)\\
        &= \frac{m(t,x) m(t+x,-x)}{m(x,-x)} f(t+x+H) h(t).
    \end{align*}
    Using the cocycle relation, we see that $\frac{m(t,x) m(t+x,-x)}{m(x,-x)} = 1$, hence the claim follows.
\end{proof}
\begin{lem}
    Let $f \in L^\infty(\Xi / H)$. Then, $M_f \in \mathcal C_1(L^2(\Xi \sslash H))$ if and only if $f \in \operatorname{BUC}(\Xi / H)$.
\end{lem}
\begin{proof}
    Follows from the previous lemma and
    \begin{align*}
        \| \alpha_x(M_f) - M_f\| = \| M_{f(\cdot + x) - f}\| = \| f(\cdot + x) - f\|_\infty.
    \end{align*}
    Now, $\| f(\cdot + x) - f\|_\infty \to 0$ as $x \to 0$ if and only if $f\in\operatorname{BUC}(\Xi / H)$.
\end{proof}
As an immediate consequence of the lemma, we obtain:
\begin{thm}
    The following holds true: $\mathcal C_1(L^2(\Xi \sslash H))_H \cong \operatorname{BUC}(\Xi / H)$.
\end{thm}
As the last part of this section, we complete the proofs of our main results:
\begin{proof}[Proof of Theorem \ref{thm:main} and Corollary \ref{corollary}]
To finish the proof of both the theorem and its corollary, we need to show that the representation on $L^2(\Xi \sslash H)$ is irreducible. When this is proven, an application of the Stone-von Neumann theorem finishes the proof.

Since the representation on $L^2(\Xi \sslash H)$ is unitarily equivalent with the one on $L^2(\Xi / H)$, we may instead prove irreducibility there. Let us denote $\widetilde{U}_x = A_\Phi^\ast U_x A_\Phi$. By Schur's lemma, irreducibility is equivalent to the statement:
\begin{align*}
    \{ A \in L^2(\Xi / H): ~\widetilde{U}_x A = A\widetilde{U}_x\} = \mathbb C I.
\end{align*}
This is what we are going to prove. By Proposition \ref{prop:multiplicationoperatorsintertwine}, this class is contained in the set of all multiplication operators. Hence, let $f \in L^\infty(\Xi/H)$ such that $\widetilde{U}_x M_f = M_f \widetilde{U}_x$ for every $x \in \Xi$. Writing out the definition, this implies that for every $g \in L^2(\Xi / H)$, for every $x \in \Xi$ and $t+H$-a.e.:
\begin{align*}
    f(t+H) &\frac{\Phi(t+x)}{\phi(t)} m(t,x) g(t+x+H) = M_f \widetilde{U}_x g(t+H) \\
    &= \widetilde{U}_x M_f g(t+H) = \frac{\Phi(t+x)}{\Phi(t)}m(t,x) g(t+x+H) f(t+x+H).
\end{align*}
Canceling out the unimodular terms, this yields
\begin{align*}
    f(t+H) g(t+x+H) = f(t+x+H) g(t+x+H),
\end{align*}
for every $g \in L^2(\Xi / H)$, every $x \in \Xi$ and almost every $t +H \in \Xi / H$. This clearly implies $f(t+H) = f(t+x+H)$ for every $x \in \Xi$ and almost-every $t+H \in \Xi / H$, which in turn implies that $f = const$ almost everywhere. Hence, the operator $M_f$ is contained in $\mathbb C I$, which is what we needed to prove.
\end{proof}

\section{An operator-valued version of the theorem}\label{sec:3}
For a locally compact abelian group $G$ and any Hilbert space $\mathcal K$, we denote by $L^2(G; \mathcal K)$ the $\mathcal K$-valued square-integrable functions on $G$. Note that this can be identified with the Hilbert space tensor product $L^2(G) \widehat{\otimes} \mathcal K$. Further, by $L^\infty(G; \mathcal L(\mathcal K))$, we denote the $\mathcal L(\mathcal K)$-valued essentially bounded functions on $G$. For $f \in L^\infty(G; \mathcal L(\mathcal K))$ and $g \in L^2(G; \mathcal K)$, the multiplication operator $M_f$ on $\mathcal L(L^2(G; \mathcal K))$ is defined as $M_f g(x) = f(x)(g(x))$. 

Given a phase space $(\Xi, m)$ and its irreducible unitary $m$-representation ($\mathcal H, U$), we define a projective representation on $\mathcal H \widehat{\otimes} \mathcal K$ by defining it on elementary tensors:
\begin{align*}
    U_x'(\varphi \otimes \psi) = U_x \varphi \otimes \psi, ~\varphi \in \mathcal H, \psi \in \mathcal K.
\end{align*}
Given this, we can consider (where, as usual, $H$ is a closed subgroup of $\Xi$):
\begin{align*}
    \mathcal L(\mathcal H \widehat{\otimes} \mathcal K)_H = \{ A \in \mathcal L(\mathcal H \widehat{\otimes} \mathcal K): ~U_h' A = A U_h' \text{ for every } h \in H\}.
\end{align*}
As a variation of Theorem \ref{thm:main}, we are going to prove the following result:
\begin{thm}\label{thm:matrix}
    Let $(\Xi, m)$ be an abelian phase space, $(\mathcal H, U)$ an irreducible $m$-projective unitary representation of $\Xi$ and $\mathcal K$ another Hilbert space. Further, let $H \subset \Xi$ be a first-countable Lagrangian subgroup. Then,
    \begin{align*}
        \mathcal  L(\mathcal H \otimes \mathcal K)_H \cong L^\infty(\Xi / H; \mathcal K).
    \end{align*}
\end{thm}
\begin{proof}
    Let $(e_j)_{j \in \Gamma}$ be an orthonormal basis of $\mathcal K$ (where $\Gamma$ is an appropriate index set). Then, for $A \in \mathcal L(\mathcal H \otimes \mathcal K)$, we can define $A_{j,k} \in \mathcal L(\mathcal H)$ by:
    \begin{align*}
        \langle A_{j,k}\varphi , \phi\rangle_{\mathcal H} := \langle A (\varphi \otimes e_j), (\phi \otimes e_k)\rangle_{\mathcal H \widehat{\otimes} \mathcal K}. 
    \end{align*}
    As a matter of fact, $A \in \mathcal L(\mathcal H \otimes \mathcal K)_H$ if and only if $A_{j,k} \in \mathcal L(\mathcal H)_H$ for every pair $j, k \in \Gamma$. 

    Let $V: \mathcal H \to L^2(\Xi / H)$ be the unitary operator intertwining the representations, obtained from the theorem of Stone and von Neumann. Then, the operator $V': \mathcal H \widehat{\otimes} \mathcal K \to L^2(\Xi / H; \mathcal K) \cong L^2(\Xi / H) \widehat{\otimes} \mathcal K$, defined by
    \begin{align*}
        V'(\varphi \otimes \psi) = V(\varphi) \otimes \psi,
    \end{align*}
    intertwines the representations $U_x'$. If $A_{j,k} \in \mathcal L(\mathcal H)_H$ for every $j, k\in \Gamma$, then $[V' A (V')^\ast]_{j,k} = V A_{j,k} V^\ast$. Hence, $A \in \mathcal L(\mathcal H \widehat{\otimes} \mathcal K)_H$ if and only if $V A_{j,k}V^\ast \in L^2(\Xi / H)_H$ for every $j, k \in \Gamma$ if and only if $VA_{j,k}V^\ast = M_{f_{j,k}}$ for some $f_{j,k} \in L^\infty(\Xi / H)$. Thus, $A \in \mathcal L(\mathcal H \otimes \mathcal K)_H$ if and only if $V' A (V')^\ast = M_F$ with $F \in L^\infty(\Xi / H; ~\mathcal L(\mathcal K))$, where $\langle F(x)e_j, e_k\rangle = f_{j,k}(x)$ almost everywhere. 
\end{proof}

\section{Examples}\label{sec:examples}
We end by providing a list of examples to which our result applies. 
\begin{ex}
    Let $\Xi = \mathbb R^{n} \times \mathbb R^n$ and $m((x, \xi), (y, \eta)) = e^{-iy\xi}$. Further, let $H \subset \Xi$ be a Lagrangian subgroup. Then, up to applying a symplectomorphism, $H$ is of the form 
    \begin{align*}
        H = [\mathbb R \otimes \{ 0\}]^k \times [\sqrt{\pi}\mathbb Z \times \sqrt{\pi}\mathbb Z]^{n-k},
    \end{align*}
    cf.~\cite[Theorem 3.16]{Fulsche_Rodriguez2023}. Hence, we see that
    \begin{align*}
        \mathcal L(\mathcal H)_H \cong L^\infty(\mathbb R^k \times \mathbb T^{n-k}),\\
        \mathcal C_1(\mathcal H)_H \cong \operatorname{BUC}(\mathbb R^k \times \mathbb T^{n-k}),
    \end{align*}
    reproducing the results from \cite{Fulsche_Rodriguez2023}.
\end{ex}
\begin{ex}
    Let $\mu$ denote the Gaussian measure $d\mu(z) = \frac{1}{\pi}e^{-|z|^2}~dz$ on $\mathbb C$. For $n \in \mathbb N$, the polyanalytic Fock space $F_n^2$ is defined as the space of all smooth functions $f \in L^2(\mathbb C, \mu)$ satisfying $\frac{\partial^n f}{\partial \overline{z}^n} = 0$. This space is a frequently studied reproducing kernel Hilbert space with an interesting operator theory, we refer to \cite{Lee-Guzman2024, Fulsche_Hagger2024} and references therein for details. For $z \in \mathbb C$ we define the Weyl operators $W_z$ on $F_n^2$ defined by
    \begin{align*}
        W_z f(w) = f(w-z) e^{w \cdot \overline{z} - \frac{|z|^2}{2}}.
    \end{align*}
    These operators form a projective unitary representation of $\Xi = \mathbb C \cong \mathbb R \times \mathbb R$ and fall within the framework described in this paper. Indeed, the representation $(F_n^2, W)$ of $\mathbb C$ is not irreducible, but can be written as the direct sum of $n$ irreducible subrepresentations: $F_n^2 = F_{(1)}^2 \oplus \dots \oplus F_{(n)}^2$, cf.~\cite[Section 3]{Fulsche_Hagger2024} for details (the spaces $F_{(j)}^2$ are the so-called \emph{true polyanalytic Fock spaces}). By Stone-von Neumann, all of the irreducible representations $(F_{(j)}^2, W)$ are unitarily equivalent. In particular, $F_n^2 \cong F_{(1)}^2 \otimes \mathbb C^n$. Hence, as a consequence of Theorem \ref{thm:matrix}, for any Lagrangian subgroup $H \subset \mathbb C$ we have:
    \begin{align*}
        \mathcal L(F_n^2)_H \cong L^\infty(\mathbb C / H; \mathbb C^{n \times n}).
    \end{align*}
    When $H$ is a Lagrangian subspace, this reproduces \cite[Theorem 10.2]{Lee-Guzman2024}. Of course, by our method we do not obtain the explicit formulas for the isomorphism discussed in \cite{Lee-Guzman2024}.
\end{ex}
\begin{ex}
    Let $\Xi = \mathbb Z \times \mathbb T$ and $m((k, \theta), (m, \vartheta)) = \theta^{-m}$ for $(k, \theta), (m, \vartheta) \in \Xi$. Then, both $\mathbb Z \times \{ 1\}$ and $\{ 0\} \times \mathbb T$ are clearly Lagrangian subgroups. Generalizing the first example, for each $\vartheta \in \mathbb T$ the subgroup $H_\vartheta = \{ (k, \vartheta^k): ~k \in \mathbb Z\}$ is Lagrangian. In this case, the map $(m, \theta) + H_\vartheta \mapsto \theta \vartheta^{-m}$ establishes the isomorphism $\Xi / H_\vartheta \cong \mathbb T$ such that $\mathcal L(\mathcal H)_{H_\vartheta} \cong L^\infty(\mathbb T)$. Further, $\mathcal L(\mathcal H)_{\{ 0\} \times \mathbb T} \cong \ell^\infty(\mathbb Z)$.
\end{ex}
\begin{ex}
    Let $p$ be a prime number and $\mathbb Q_p$ denote the additive group of $p$-adic numbers, topologized as the completion of $\mathbb Q$ with respect to the $p$-adic norm. Then, each $x \in \mathbb Q_p$ can be written as $x = \sum_{j=m}^\infty x_j p^j$, where $m \in \mathbb Z$ and $x_j \in \{ 0, 1, \dots, p-1\}$. The dual group $\widehat{\mathbb Q_p}$ can be identified with $\mathbb Q_p$ via
    \begin{align*}
        \mathbb Q_p \ni y \mapsto \xi_y, \quad \xi_y(x) = e^{2\pi i x\cdot y}.
    \end{align*}
    Here, for any $x \in \mathbb Q_p$, one understands $e^{2\pi i x}$ as
    \begin{align*}
        e^{2\pi i x} = e^{2\pi i \sum_{j=m}^\infty x_j p^j} = e^{2\pi i \sum_{j=m}^{-1} x_j p^j},
    \end{align*}
    where we use that $e^{2\pi i x_j p^j} = 1$ for every $j \geq 0$. Hence, $\Xi = \mathbb Q_p \times \mathbb Q_p$ is a phase space, when endowed with the multiplier $m((x, y), (x', y')) = e^{2\pi i x' \cdot y}$. Clearly, both $\mathbb Q_p\times \{ 0\}$ and $\{ 0\} \times \mathbb Q_p$ are Lagrangian subgroups with $\Xi /(\mathbb Q_p\times \{ 0\}) \cong \mathbb Q_p \cong \Xi / (\{ 0\} \times \mathbb Q_p)$. Besides this, another example of a Lagrangian subgroup is the $p$-adic lattice $H := \mathbb Z_p \times \mathbb Z_p$, where $\mathbb Z_p$ denotes the $p$-adic integers: Every $x\in \mathbb Q_p$ with $x_j = 0$ for $j < 0$. It is not hard to verify that $H = \mathcal H^\sigma$. Hence, $\mathcal L(\mathcal H)_H \cong L^\infty(\Xi / H)$. Note that $\mathbb Q_p / \mathbb Z_p \cong \mathbb Z(p^\infty)$ is the Pr\"{u}fer group such that $\Xi / H \cong \mathbb Z(p^\infty) \times \mathbb Z(p^\infty)$. Since the Pr\"{u}fer group is discrete, we obtain that $\mathcal L(\mathcal H)_H = \mathcal C_1(\mathcal H)_H$ for $H$ being the $p$-adic lattice.
\end{ex}
For the last example, we want to recall one of the results from \cite{Fulsche_Rodriguez2023}. As already mentioned before, there the discussions were made for $\Xi = \mathbb R^{2n}$ with the standard symplectic form, but everything mentioned there holds true in the general case, as long as the unique irreducible representation is square integrable (see \cite{Fulsche_Galke2023} for a discussion on general QHA on abelian phase spaces).

As in \cite[Corollary 3.4]{Fulsche_Rodriguez2023} one can prove:
\begin{thm}
    Let $(\Xi, m)$ be an abelian phase space such that the unique irreducible $m$-representation $(U_x)_{x \in \Xi}$ is square-integrable. Then, for a closed subgroup $H \subset \Xi$ it is
    \begin{align*}
        \mathcal L(\mathcal H)_H = \overline{\operatorname{span}} \{ U_x: ~x\in H^\sigma\},
    \end{align*}
    where the closure is taken in weak$^\ast$ topology.
\end{thm}
The previous result can be used to generate a wealth of commutative matrix algebras, by applying it to finite phase spaces. Note that, in this case, the representing space $\mathcal H$ is finite dimensional, so that taking the closure in the above result is of course redundant.

    Upon applying the presented theory to finite abelian groups, it can be used to construct a wealth of commutative matrix algebras. To keep the examples computationally simple, we stick with rather small groups. 

Recall that for a finite abelian group $G$ it is $\widehat{G} \cong G$ (even though the isomorphism is not canonical). Further, any finite abelian group can be described as the finite product of finite cyclic groups. Since products of groups yield tensor product of commutative matrix algebras, we restrict our examples to the irreducible cases. In the following, we will always endow $\Xi = G \times \widehat{G}$ with the multiplier $m((x, \varphi), (y, \psi)) = \overline{\varphi(y)}$ such that the symplectic form equals $\sigma((x, \varphi), (y, \psi)) = \psi(x)/\varphi(y)$.

\begin{ex}
For $G = Z_2 = \{ 0, 1\}$, the dual group consists of the two characters $\varphi_0 = 1$ and $\varphi_1(1)  = -1$. On $\ell^2(Z_2) \cong \mathbb C^2$ the phase space is represented by
\begin{align*}
    U_{(0, \varphi_0)}  &\cong \begin{pmatrix}
        1 & 0 \\ 0 & 1
    \end{pmatrix}, \quad U_{(1, \varphi_0)} \cong \begin{pmatrix}
        0 & 1\\ 1 & 0
    \end{pmatrix},\\
    U_{(0, \varphi_1)} &\cong \begin{pmatrix}
        1 & 0 \\ 0 & -1
    \end{pmatrix}, \quad U_{(1, \varphi_1)} \cong \begin{pmatrix}
        0 & -1 \\ 1 & 0
    \end{pmatrix}. 
\end{align*}
It is not hard to figure out that the only Lagrangian subgroups of $\Xi$ are $H_1 = Z_2 \times \{ \varphi_0\}, H_2 = \{ 0 \} \times \widehat{Z_2}$ and $H_3 = \{ (0, \varphi_0), (1, \varphi_1)\}$. For these cases, the commutative matrix algebras are:
\begin{align*}
    \mathcal L(\mathcal H)_{H_1} &\cong \left \{ \begin{pmatrix}
        a & b \\ b & a
    \end{pmatrix}: ~a, b \in \mathbb C \right \}, \\
    \mathcal L(\mathcal H)_{H_2} &\cong \left \{ \begin{pmatrix}
        a & 0\\ 0 & b
    \end{pmatrix} : ~a, b \in \mathbb C \right \},\\
    \mathcal L(\mathcal H)_{H_3} &\cong \left \{ \begin{pmatrix}
        a & b\\ -b & a
    \end{pmatrix}: ~a, b \in \mathbb C \right \}.
\end{align*}
Of course, each of them has Gelfand spectrum consisting of two points.
\end{ex}
\begin{ex}
    For $G = Z_3 = \{ 0, 1, 2\}$, the dual group consists of the characters $\varphi_0 = 1$, $\varphi_1(j) = e^{2\pi i j/3}$ and $\varphi_2(j) = e^{-2\pi i / 3}$. The representing operators on $\ell^2(Z_3) \cong \mathbb C^3$ acting by $U_{(j, \varphi_k))} a(n) = \varphi_k(n) a(n-j)$ can be identified with the matrices:
    \begin{align*}
        U_{(0, \varphi_0)} &\cong \begin{pmatrix}
            1 & 0 & 0\\
            0 & 1 & 0\\
            0 & 0 & 1
        \end{pmatrix}, \quad 
        U_{(1, \varphi_0)} \cong \begin{pmatrix}
            0 & 0 & 1\\
            1 & 0 & 0\\
            0 & 1 & 0
        \end{pmatrix}, \quad 
        U_{(2, \varphi_0)} \cong \begin{pmatrix}
            0 & 1 & 0\\
            0 & 0 & 1\\
            1 & 0 & 0
        \end{pmatrix},\\
        U_{(0, \varphi_1)} &\cong \begin{pmatrix}
            1 & 0 & 0\\
            0 & e^{2\pi i/3} & 0\\
            0 & 0 & e^{-2\pi i/3}
        \end{pmatrix}, \quad 
        U_{(1, \varphi_1)} \cong \begin{pmatrix}
            0 & 0 & e^{-2\pi i/3}\\
            1 & 0 & 0\\
            0 & e^{2\pi i/3} & 0
        \end{pmatrix},\\
        U_{(2, \varphi_1)} &\cong \begin{pmatrix}
            0 & e^{2\pi i/3} & 0\\
            0 & 0 & e^{-2\pi i/3}\\
            1 & 0 & 0
        \end{pmatrix},\quad
        U_{(0, \varphi_2)} \cong \begin{pmatrix}
            1 & 0 & 0\\
            0 & e^{-2\pi i/3} & 0\\
            0 & 0 & e^{2\pi i/3}
        \end{pmatrix},\\ 
        U_{(1, \varphi_2)} &\cong \begin{pmatrix}
            0 & 0 & e^{2\pi i/3}\\
            1 & 0 & 0\\
            0 & e^{-2\pi i/3} & 0
        \end{pmatrix}, \quad 
        U_{(2, \varphi_2)} \cong \begin{pmatrix}
            0 & e^{-2\pi i/3} & 0\\
            0 & 0 & e^{2\pi i/3}\\
            1 & 0 & 0
        \end{pmatrix}.
    \end{align*}
    Further, it is not hard to determine that the Lagrangian subgroups of $\Xi$ are:
    $H_1 = \{ 0, 1, 2\} \times \{ \varphi_0\}$, $H_2 = \{ 0\} \times \{ \varphi_0, \varphi_1, \varphi_2\}$, $H_3 = \{ (0, \varphi_0), (1, \varphi_1), (2, \varphi_2)\}$ and $H_4 = \{ (0, \varphi_0), (1, \varphi_2), (2, \varphi_1)\}$. Hence, the commutative matrix algebras obtained are:
    \begin{align*}
        \mathcal L(\mathcal H)_{H_1} &\cong \left \{ \begin{pmatrix}
            a & b & c\\ c & a & b\\ b & c & a\end{pmatrix}: ~a, b, c \in \mathbb C\right \},\\
        \mathcal L(\mathcal H)_{H_2} &\cong \left \{ \begin{pmatrix}
            a & 0 & 0\\ 0 & b & 0 \\ 0 & 0 & c
        \end{pmatrix}: ~a, b, c\in \mathbb C\right \}, \\
        \mathcal L(\mathcal H)_{H_3} &\cong \left \{ \begin{pmatrix}
            a & e^{-2\pi i/3}b & e^{-2\pi i/3}c\\ c & a & e^{2\pi i/3}c \\ b & e^{2\pi i/3}c & a
        \end{pmatrix}: ~a, b, c\in \mathbb C\right \},\\
        \mathcal L(\mathcal H)_{H_4} &\cong \left \{ \begin{pmatrix}
            a & e^{2\pi i/3}b & e^{2\pi i/3}c\\ c & a & e^{-2\pi i/3}b \\ b & e^{-2\pi i/3}c & a
        \end{pmatrix}: ~a, b, c\in \mathbb C\right \}.
    \end{align*}
    Of course, each of those four algebras has Gelfand spectrum consisting of three points.
\end{ex}

\bibliographystyle{amsplain}
\bibliography{References}

\bigskip
\begin{multicols}{2}

\noindent
Robert Fulsche\\
\href{fulsche@math.uni-hannover.de}{\Letter ~fulsche@math.uni-hannover.de}
\\
\noindent
Institut f\"{u}r Analysis\\
Leibniz Universit\"at Hannover\\
Welfengarten 1\\
30167 Hannover\\
Germany

\noindent
Oliver F\"{u}rst\\
\href{ofuerst@math.uni-bonn.de}{\Letter ~ofuerst@math.uni-bonn.de}
\\
\noindent
Mathematisches Institut\\
Universit\"{a}t Bonn\\
Endenicher Allee 60\\
53115 Bonn\\
Germany

\end{multicols}

\end{document}